\documentclass{aims}
\usepackage{amsmath}
  \usepackage{paralist}
  \usepackage{graphics} 
  \usepackage{epsfig} 
\usepackage{graphicx}  \usepackage{epstopdf}
 \usepackage[colorlinks=true]{hyperref}
\hypersetup{urlcolor=blue, citecolor=red}

\def\currentvolume{X}
 \def\currentissue{X}
  \def\currentyear{200X}
   \def\currentmonth{XX}
    \def\ppages{X--XX}
     \def\DOI{10.3934/xx.xx.xx.xx}
     
\newtheorem{theorem}{Theorem}[section]

\newtheorem{lemma}[theorem]{Lemma}

\theoremstyle{definition}
\newtheorem{definition}[theorem]{Definition}

\newtheorem{notation}{Notation}

\def\clo#1{\overline{#1}}
\def\text#1{\mbox{#1}}

\def\nullsp#1{\text{{\textit{null}}($#1$)}}

\newcommand{\ra}{\rangle}
\newcommand{\la}{\langle}

\newcommand{\Sop}{\mathcal{S}}
\newcommand{\Cop}{\mathcal{C}}
\newcommand{\Kop}{\mathcal{K}}
\newcommand{\Sph}{\mathbb{S}}

\newcommand{\Gpm}{(\partial \Omega \times \Sph)_{\pm}}

\newcommand{\Ws}{\mathbb{V}^{0}}
\newcommand{\Wss}{\mathbb{V}^{1}}
\newcommand{\Ts}{\mathbb{T}}

\def\Ls#1{\mathbb{W}^{0,#1}}
\def\Lss#1{\mathbb{W}^{1,#1}}
\def\LTs#1{\mathbb{T}^{0,#1}}

\title[Recovery of the absorption coefficient]{Recovery of the absorption coefficient in radiative transport from a single measurement}

\author[Sebastian Acosta]{}


 \email{sacosta@bcm.edu}

\thanks{This work was partially supported by AFOSR Grant
FA9550-12-1-0117 and ONR Grant N00014-12-1-0256}

\begin{document}

\maketitle

\centerline{\scshape Sebastian Acosta}

\medskip

{\footnotesize
 \centerline{Department of Pediatric Cardiology}
   \centerline{Baylor College of Medicine, Houston, TX, USA}
}

\bigskip

 \centerline{(Communicated by the associate editor name)}

\begin{abstract}
In this paper, we investigate the recovery of the absorption coefficient from boundary data assuming that the region of interest is illuminated at an initial time. We consider a sufficiently strong and isotropic, but otherwise \textit{unknown} initial state of radiation. This work is part of an effort to reconstruct optical properties using \textit{unknown} illumination embedded in the \textit{unknown} medium.

We break the problem into two steps. First, in a linear framework, we seek the simultaneous recovery of a forcing term of the form $\sigma(t,x,\theta) f(x)$ (with $\sigma$ known) and an \textit{isotropic} initial condition $u_{0}(x)$ using the \textit{single measurement} induced by these data. Based on exact boundary controllability, we derive a system of equations for the unknown terms $f$ and $u_{0}$. The system is shown to be Fredholm if $\sigma$ satisfies a certain positivity condition. We show that for generic term $\sigma$ and weakly absorbing media, this linear inverse problem is uniquely solvable with a stability estimate. In the second step, we use the stability results from the linear problem to address the nonlinearity in the recovery of a weak absorbing coefficient. We obtain a locally Lipschitz stability estimate.
\end{abstract}


\section{Introduction} \label{Section.Intro}

The radiative transport equation (RTE) models physical phenomena in various scientific disciplines including medical imaging, semiconductors, astrophysics, nuclear reactors, etc. The mathematical treatment of some of these applications is found in \cite{Cas-Zwe-1967,Dau-Lio-1993,Agoshkov-1998,Mok-1997,Anikonov-Book-2002,Cer-Gab-2007}. We are mainly motivated by medical imaging, and we refer the reader to \cite{Arr-1999,Gib-Heb-Arr-2005,Kim-Mos-2006,Bal-2009-Rev,Dur-Cho-Bak-Tod-2010,Ren-2010,Bal-Ren-2008,Bal-Pinaud-2007,Bal-Pinaud-2011,Arr-Schot-2009,Arr-2011} for descriptions of such problems. 

The transport of radiation is modeled by the following equations,
\begin{align}
\partial_{t} w + (\theta \cdot \nabla) w + \mu_{\rm a} w + \mu_{\rm s} (I - \Kop) w &= 0  \quad && \text{in $[0,\tau] \times (\Omega \times \Sph) $} \label{Eqn.Orig01} \\
w &= w_{0} \quad && \text{on $\{ t = 0 \} \times (\Omega \times \Sph)$}, \label{Eqn.Orig02}
\end{align}
augmented by a vanishing inflow profile. The optical properties of the medium are the absorption coefficient $\mu_{\rm a}$, the scattering coefficient $\mu_{\rm s}$ and the scattering operator $\Kop$.  The solution $w=w(t,x,\theta)$ represents the density of radiation at time $t \in [0,\tau]$, position $x \in \Omega \subset \mathbb{R}^n$, moving in the direction $\theta \in \Sph$ at unit speed.

An inverse problem for optical imaging consists of reconstructing one or more optical coefficients. In this paper, we focus on the recovery of the absorption coefficient $\mu_{\rm a}$ and we assume that the scattering properties of the medium are known. We let $\tilde{w}$ have vanishing inflow profile and be the solution of the following radiative transport problem with absorption coefficient $\tilde{\mu}_{\rm a}$ and initial condition $\tilde{w}_{0}$,
\begin{align}
\partial_{t} \tilde{w} + (\theta \cdot \nabla) \tilde{w} + \tilde{\mu}_{\rm a} \tilde{w} + \mu_{\rm s} (I - \Kop) \tilde{w} &= 0  \quad && \text{in $[0,\tau] \times (\Omega \times \Sph) $} \label{Eqn.Orig03} \\
\tilde{w} &= \tilde{w}_{0} \quad && \text{on $\{ t = 0 \} \times (\Omega \times \Sph)$}, \label{Eqn.Orig04}
\end{align}

Then the following functions
\begin{eqnarray}
u = w - \tilde{w}, \qquad
\sigma = \tilde{w}, \qquad
f = \mu_{\rm a} - \tilde{\mu}_{\rm a}  \qquad \text{and} \qquad
u_{0} = w_{0} - \tilde{w}_{0}  \label{Eqn.Diff}
\end{eqnarray}
satisfy the initial boundary value problem (\ref{Eqn.0005})-(\ref{Eqn.0007}) defined in Section \ref{Section:MainResults}. As usual, the system (\ref{Eqn.0005})-(\ref{Eqn.0007}) has the aspect of a linearized problem where $\mu_{\rm a}$ represents a reference absorption and we seek to reconstruct the contrast $f = \mu_{\rm a} - \tilde{\mu}_{\rm a}$ between the true and reference coefficients from knowledge of the boundary measurements of the field $u$.  

Some of the most fundamental results in the recovery of coefficients for the RTE are based on knowledge of the \textit{albedo} operator which maps \textit{all possible} inflow illuminations to corresponding outflow measurements. For reviews of these results see Stefanov \cite{Ste-InOut-2003}, Bal \cite{Bal-2009-Rev} and reference therein. In the present work, however, we focus on the recovery of a coefficient from a \textit{single measurement}. To the best of our knowledge, all works in the literature for similar single-measurement problems aim at the recovery of $f$ given full knowledge of the initial condition. Hence, it is usually assumed that $u_{0} = 0$. In this category of assumptions, we find the recent works of Klibanov--Pamyatnykh \cite{Kli-Pam-2008} and Machida--Yamamoto \cite{Mach-Yam-2013}. Both of them are based on ingenious Carleman estimates leading to uniqueness in the recovery of $f$ in \cite{Kli-Pam-2008}, and global Lipschitz stability in \cite{Mach-Yam-2013}. These estimates are patterned after similar results for the wave equation in the general form of \cite[Thm 8.2.2]{Isakov-1998}.
Unfortunately, assuming precise knowledge of a non-vanishing initial condition is not realistic in most applications. The main contribution of this paper is a strategy to partially bypass that assumption.

One of the standing questions in some optical imaging applications is the following: Can we use \textit{unknown} sources (such as ambient radiative noise or fortuitous radiance) to illuminate a region and recover its optical properties ? In this paper, the illuminating source is represented by the initial state of radiation $\sigma|_{t=0} = \tilde{w}|_{t=0}$, and we seek to show that $f=0$ if the boundary outflow is zero. The purpose of our work is to provide sufficient conditions for the unique and stable recovery of $f$ from knowledge of the outflow profile of radiation $u$ at the boundary of $\Omega$, but \textit{without} full knowledge of $u_{0}$. In plain words, our main result is that if the initial condition $u_{0}$ is \textit{isotropic} and the region $\Omega$ is properly illuminated, then both $f$ and $u_{0}$ can be reconstructed in a locally stable manner. The precise statements of our results are presented in Section \ref{Section:MainResults}. Our proof, found in Section \ref{Section:InverseProblem}, is primarily based on boundary controllability for the RTE recently obtained in \cite{Aco-2013,Kli-Yam-2007}.

 
\section{Background and statement of main results} \label{Section:MainResults}

In this section we state the direct problem for transient radiative transport and also the exact boundary controllability property. We also review some preliminary facts in order to state our main results in the proper mathematical ground. 

We assume that $\Omega \subset \mathbb{R}^n$ ($n \geq 2$) is a bounded convex domain with smooth boundary $\partial \Omega$. The unit sphere in $\mathbb{R}^n$ is denoted by $\Sph$. The outflow $(+)$ and inflow $(-)$ parts of the boundary are
\begin{eqnarray*}
\Gpm = \{ (x,\theta) \in \partial \Omega \times \Sph \, : \, \pm \, \nu(x) \cdot \theta > 0 \} 
\end{eqnarray*}
where $\nu$ is the outward unit normal vector on $\partial \Omega$. Without loss of generality, it is assumed that the particles travel at unit speed. The spatial scale of the problem is given by
\begin{eqnarray*}
l = \text{\rm diam}(\Omega).
\end{eqnarray*}

Now we define the appropriate Banach spaces over which the radiative transport problem is well-posed. First, we denote by $\Ls{p}$ and $\Lss{p}$, for $0 \leq p < \infty$, the completion of $C^{1}(\clo{\Omega} \times \Sph)$ with respect to the following norms,
\begin{eqnarray}
&& \| u \|_{\Ls{p}} =  \| u \|_{L^p(\Omega \times \Sph)} \\ \label{Eqn.146}
&& \| u \|_{\Lss{p}}^{p} = l^{p} \| \theta \cdot \nabla u \|_{\Ls{p}}^{p} + \| u \|_{\Ls{p}}^{p} +  l \| |\nu \cdot \theta|^{1/p} u \|_{L^p(\partial \Omega \times \Sph)}^{p}  \label{Eqn.148}
\end{eqnarray}
where $\theta \cdot \nabla$ denotes the weak directional derivative. These expressions adopt the usual meaning for $p=\infty$. Now, denote by $\LTs{p}$ the trace space defined as the completion of $C(\partial \Omega \times \Sph)$ with respect to the following norm,
\begin{eqnarray}
&& \| u \|_{\LTs{p}} = l \| |\nu \cdot \theta|^{1/p} u \|_{L^p(\partial \Omega \times \Sph)}. \label{Eqn.150}
\end{eqnarray}
We also have the spaces $\LTs{p}_{\pm}$ denoting the restriction of functions in $\LTs{p}$ to the in- and out-flow portions of the boundary $\partial \Omega \times \Sph$, respectively. Functions in $\Lss{p}$ have well-defined traces on the space $\LTs{p}$ as asserted by the following lemma whose proof is found in \cite{Egg-Sch-2014,Agoshkov-1998,Mok-1997,Ces-1985,Ces-1984}.

 
\begin{lemma} \label{Lemma.101}
The trace mapping $u \mapsto u|_{\partial \Omega}$ defined for $C^{1}(\clo{\Omega} \times \Sph)$ can be extended to a bounded operator $\gamma : \Lss{p} \to \LTs{p}$. Moreover, $\gamma : \Lss{p} \to \LTs{p}$ is surjective. The same claims hold for the partial trace maps $\gamma_{\pm} : \Lss{p} \to \LTs{p}_{\pm}$.
\end{lemma}

In addition, we have the following definition for traceless closed subspaces of $\Lss{p}$,
\begin{eqnarray}
\Lss{p}_{\pm} = \nullsp{\gamma_{\pm}} = \left\{ v \in \Lss{p} \, : \, \gamma_{\pm} v = 0 \right\}, \label{Eqn.003}
\end{eqnarray}
as well as the following integration-by-parts formula or Green's identity for functions $u,v \in \Lss{p}$ for $p \geq 2$,
\begin{eqnarray}
\int_{\Omega \times \Sph} \!\!\!\! (\theta \cdot \nabla u) v  = \int_{\partial \Omega \times \Sph} \!\!\!\! (\theta \cdot \nu) u v  - \int_{\Omega \times \Sph} \!\!\!\! (\theta \cdot \nabla v) u. \label{Eqn.Green}
\end{eqnarray}

\begin{notation}
In order to ease the notation, we define special symbols for the following Hilbert spaces $\Ws = \Ls{2}$, $\Wss = \Lss{2}$ and $\Ts = \LTs{2}$.
\end{notation}

Throughout the paper we will make the following assumptions concerning the absorption and scattering coefficients, and the scattering kernel. These assumptions allow heterogeneous media to be modeled by coefficients of low regularity. First, we have non-negative absorption $\mu_{\rm a} \in L^{\infty}(\Omega)$ and scattering $\mu_{\rm s} \in L^{\infty}(\Omega)$ coefficients. We denote their norms as follows,
\begin{eqnarray}
\clo{\mu}_{\rm a} = \| \mu_{\rm a} \|_{L^{\infty}(\Omega)} \qquad \text{and} \qquad \clo{\mu}_{\rm s} = \| \mu_{\rm s} \|_{L^{\infty}(\Omega)}.  \label{Eqn.AbsScaCoeff}
\end{eqnarray}
The scattering operator $\Kop : \Ls{p} \to \Ls{p}$ is given by
\begin{eqnarray}
 (\Kop u)(x,\theta) = \int_{\Sph} \kappa (x,\theta,\theta') u(x,\theta') \, d S(\theta'), \label{Eqn.009}
\end{eqnarray}
where the scattering kernel $0 \leq \kappa \in L^{\infty}(\Omega \times \Sph \times \Sph)$. It is assumed that the scattering operator is \textit{conservative} in the following sense,
\begin{eqnarray}
&& \int_{\Sph} \kappa(x,\theta,\theta') d S(\theta') = 1, \quad \text{for a.a. $(x,\theta) \in \Omega \times \Sph$}. \label{Eqn.011}
\end{eqnarray}
In addition, we assume a \textit{reciprocity condition} on the scattering kernel given by
\begin{eqnarray}
&& \kappa(x,\theta,\theta') = \kappa(x,-\theta' , -\theta ), \quad \text{for a.a. $(x,\theta,\theta') \in \Omega \times \Sph \times \Sph$}. \label{Eqn.013}
\end{eqnarray}
This means that the scattering events are reversible in a local sense at each point $x \in \Omega$. The direct problem is defined as follows.

\begin{definition}[Direct Problem] \label{Def.DirectProb2}
Given forcing terms
\begin{eqnarray}
f \in L^{\infty}(\Omega) \quad \text{and} \quad
 \sigma \in C^{1}([0,\tau]; L^{\infty}(\Omega \times \Sph)), \nonumber
\end{eqnarray}
and initial condition
\begin{eqnarray}
u_{0} \in \Lss{\infty}_{-}, \nonumber
\end{eqnarray}
find a solution $u \in C^{1}([0,\tau]; \Ls{\infty}) \cap C([0,\tau]; \Lss{\infty})$ to the following initial boundary value problem
\begin{align}
\dot{u} + (\theta \cdot \nabla) u + \mu_{\rm a} u + \mu_{\rm s} (I - \Kop) u &= \sigma f  \quad && \text{in $[0,\tau] \times (\Omega \times \Sph) $}, \label{Eqn.0005} \\
u &= u_{0} \quad && \text{on $\{ t = 0 \} \times (\Omega \times \Sph)$}, \label{Eqn.0006} \\
\gamma_{-} u &= 0 \quad && \text{on $[0,\tau] \times (\partial \Omega \times \Sph)_{-}$}. \label{Eqn.0007}
\end{align}
\end{definition}

From the well-posedness of the stationary radiative transport problem \cite{Egg-Sch-2014,Agoshkov-1998,Mok-1997} and semigroup theory \cite{Eng-Nag-2000,Pazy-1983}, we obtain that the above transient radiative transport problem for general heterogeneous, scattering media is well-posed.


\subsection{Tools from control theory} \label{Subsec:MainResultControl}

Here we proceed to define the boundary controllability for the RTE. This is the major tool to prove our main result. Hence, we consider the following \textit{adjoint} transport problem with prescribed \textit{outflow} data. Given $\eta \in L^2([0,\tau];\Ts_{+})$, find a mild solution $\psi \in  C([0,\tau]; \Ws)$ of the following problem
\begin{eqnarray}
\dot{\psi} + (\theta \cdot \nabla) \psi - \mu_{\rm a} \psi - \mu_{\rm s} (I - \Kop^{*}) \psi = 0 \quad &\text{in $[0,\tau] \times (\Omega \times \Sph)$}, \label{Eqn.001c} \\
\psi = 0 \quad &\text{on $\{ t = \tau \} \times (\Omega \times \Sph)$}, \label{Eqn.002c} \\
\gamma_{+} \psi = \eta \quad &\text{on $[0,\tau] \times (\partial \Omega \times \Sph)_{+}$}. \label{Eqn.003c}
\end{eqnarray}

Given arbitrary $\phi \in \Ws$, the goal of the control problem is to find an outflow control condition $\eta \in L^2([0,\tau];\Ts_{+})$ to drive the solution $\psi$ of (\ref{Eqn.001c})-(\ref{Eqn.003c}) from  $\psi(\tau) = 0$ to $\psi(0) = \phi$. The well-posedness of the control problem is described in the following theorem which is a direct consequence of \cite{Aco-2013}.

\begin{theorem} \label{Thm.MainControl2}
Assume that $l \overline{\mu}_{\rm s} \, e^{l ( \overline{\mu}_{\rm a} + \overline{\mu}_{\rm s} ) }   < e^{-1}$. Then there exists a steering time $\tau < \infty$ such that for any initial state $\phi \in \Ws$, there exists outflow control $\eta \in L^2([0,\tau];\Ts_{+})$ so that the mild solution $\psi \in C([0,\tau];\Ws)$ of the problem (\ref{Eqn.001c})-(\ref{Eqn.003c}) satisfies $\psi(0) = \phi$.
Among all such controls there exists $\eta_{\rm min}$, with minimum norm, which is uniquely determined by $\phi$ and satisfies the following stability condition
\begin{eqnarray*}
 \| \eta_{\rm min} \|_{L^{2}([0,\tau];\Ts_{+})} \leq C \| \phi \|_{\Ws}
\end{eqnarray*}
for some positive constant $C = C(\overline{\mu}_{a}, \overline{\mu}_{s}, l,\tau)$. 
\end{theorem}

Hence, the mapping $\phi \mapsto \eta_{\rm min}$ described in Theorem \ref{Thm.MainControl2} defines a bounded \textit{control operator},
\begin{eqnarray}
\Cop : \Ws \to L^{2}([0,\tau];\Ts_{+}) \label{Eqn.ControlOp}.
\end{eqnarray}

We also define a bounded \textit{solution operator}
\begin{eqnarray}
\Sop : \Ws \to L^{2}([0,\tau];\Ws) \label{Eqn.SolutionOp}
\end{eqnarray}
mapping $\phi \mapsto \psi$ where $\psi$ is the solution of (\ref{Eqn.001c})-(\ref{Eqn.003c}) with $\eta = \Cop \phi$.

We wish to point out that the condition $l \overline{\mu}_{\rm s} \, e^{l ( \overline{\mu}_{\rm a} + \overline{\mu}_{\rm s} ) }   < e^{-1}$ can be avoided using the control result of Klibanov--Yamamoto \cite{Kli-Yam-2007} if we assume sufficiently regular coefficients $\mu_{\rm a}$, $\mu_{\rm s}$ and $\kappa$.


\subsection{Main result for the inverse problem} \label{Subsec:MainResultInv}

Now we state the inverse problem for transient transport along with our main results. Our proof, presented in Section \ref{Section:InverseProblem}, is based on tools from control theory developed in \cite{Aco-2013,Kli-Yam-2007}. We break the problem into two steps. First, we study a linear problem for the simultaneous recovery of the forcing term $f \sigma$ (with $\sigma$ known) and the isotropic initial condition $u_{0}$. This linear problem can be reduced to a generically solvable Fredholm system of equations. In the second step, we use the linear results to address the nonlinear inverse problem for the recovery of the absorption coefficient. 

If we momentarily fix the source term $\sigma$, we may see the solution $u$ of the direct problem \ref{Def.DirectProb2} as dependent on the other source term $f \in L^{2}(\Omega)$ and the initial condition $u_{0} \in  \Wss_{-}$. The outflowing boundary measurements are modeled by the operator $\Lambda : L^{2}(\Omega) \times \Wss_{-} \to L^{2}( [0,\tau]; \Ts_{+})$ defined as
\begin{eqnarray}
\Lambda (f,u_{0}) = \gamma_{+} u, \label{Eqn.0630}
\end{eqnarray}
where $\gamma_{+} : \Wss \to \Ts_{+}$ is the outflowing trace operator defined in lemma \ref{Lemma.101}, and $u$ is the solution to the direct problem \ref{Def.DirectProb2}.

Notice that $\sigma f \in C^{1}([0,\tau];\Ws)$ which implies that $\dot{u} \in C([0,\tau];\Ws)$ is the mild solution of the same transport equation with $\dot{\sigma} f \in C([0,\tau];\Ws)$ as the forcing term and an initial condition in $\Ws$. For the existence of mild solutions in semigroup theory, see the standard references \cite{Eng-Nag-2000,Pazy-1983}. Now, using the concept of generalized traces for mild solutions, we can show that the measurement mapping $\Lambda : L^{2}(\Omega) \times \Wss_{-} \to H^{1}( [0,\tau]; \Ts_{+})$ is (extends as) a bounded operator. The treatment of generalized traces for mild solutions can be found in 
\cite[Section 2]{Kli-Yam-2007}, \cite[Section 14.4]{Mok-1997} or Cessenat \cite{Ces-1984,Ces-1985}.

With this notation we define both the linear and nonlinear inverse problems as follows.

\begin{definition}[Linear Inverse Problem] \label{Def.InvProb}
Let $u$ be the solution to the direct problem \ref{Def.DirectProb2} for unknown source term $f$ and unknown initial condition $u_{0}$. The inverse source problem is, given the outflowing measurement $\Lambda (f,u_{0})$, find $f$ and $u_{0}$.
\end{definition}

\begin{definition}[Nonlinear Inverse Problem] \label{Def.InvProbNonlinear}
Let $w$ and $\tilde{w}$ be the solutions to the direct problems (\ref{Eqn.Orig01})-(\ref{Eqn.Orig02}) and (\ref{Eqn.Orig03})-(\ref{Eqn.Orig04}), respectively. Show that if $\gamma_{+}w = \gamma_{+}\tilde{w}$, then $\mu_{\rm a} = \tilde{\mu}_{\rm a}$ and $w_{0} = \tilde{w}_{0}$. 
\end{definition}

In order to state and prove our main results, we define the following angular-averaging operator $P_{\theta} : \Ws \to L^2(\Omega)$ given by
\begin{eqnarray}
(P_{\theta} v)(x) = \frac{1}{|\Sph|} \int_{\Sph} v(x,\theta) \, dS (\theta),  \label{Eqn.AngAvg}
\end{eqnarray}
where $|\Sph|$ is the surface area of the unit sphere $\Sph = \{ x \in \mathbb{R}^n \, : \, |x| = 1 \}$. The well-known velocity averaging lemmas developed in \cite{Gol-Lio-Per-Sen-1988} imply the compactness of this operator when defined as $P_{\theta} : \Wss \to L^{2}(\Omega)$. We also define a time integral operator $P_{t} : L^{2}([0,\tau];\Ws) \to \Ws$ as follows,
\begin{eqnarray}
(P_{t} v)(x,\theta) = \int_{0}^{\tau} v(t,x,\theta) \, dt,  \label{Eqn.TimeAvg}
\end{eqnarray}
which is clearly bounded.

Our main result concerning the linear inverse problem is the following.

\begin{theorem}[Main Result 1] \label{Thm.MainInv}
Let $\tau < \infty$ be the steering time for exact controllability from Theorem \ref{Thm.MainControl2}. If 
\begin{itemize}
\item[(i)] there exists a constant $\delta > 0$ such that $|(P_{\theta} \sigma)(0,x)| \geq \delta$ for a.a. $x \in \Omega$,
\item[(ii)] $\sigma \in C^{1}([0,\tau];\Ls{\infty})$, and
\item[(iii)] the unknown initial condition is isotropic, ie., $u_{0} = P_{\theta} u_{0}$,
\end{itemize}
then the linear inverse problem \ref{Def.InvProb} for $(f,u_{0})$ can be reduced to the following Fredholm system on $L^{2}(\Omega) \times L^{2}(\Omega)$,
\begin{eqnarray} \left[
\begin{array}{cc}
(P_{\theta} \sigma_{0}) + (P_{\theta} P_{t} \dot{\sigma} \Sop)^{*} & \mu_{\rm a} \\ 
(P_{\theta} P_{t} \sigma \Sop)^{*} & I
\end{array} \right]
 \left[ 
\begin{array}{c}
f \\ 
u_{0}
\end{array} \right] =  \left[ 
\begin{array}{c}
P_{\theta} \, \Cop^{*} \dot{m} \\ 
P_{\theta} \, \Cop^{*} m
\end{array} \right] \label{Eqn.MainSys}
\end{eqnarray}
where $ m = \Lambda(f,u_{0})$, $\sigma_{0} = \sigma |_{t=0}$, and $P_{\theta} P_{t} \dot{\sigma} \Sop$ and $P_{\theta} P_{t} \sigma \Sop$ are compact operators.

Moreover, there exists an open and dense subset $\mathcal{U}$ of $(i) \cap (ii)$ such that for each $\sigma$ in this set, if $\clo{\mu}_{\rm a}$ is sufficiently small, then (\ref{Eqn.MainSys}) is uniquely solvable and the following stability estimate
\begin{eqnarray}
\| f \|_{L^{2}(\Omega)} + \| u_{0} \|_{L^{2}(\Omega)} \leq C \| \Lambda(f,u_{0}) \|_{H^{1}([0,\tau];\Ts_{+})} \label{Eqn.Stability}
\end{eqnarray}
holds for some positive constant $C=C(\mu_{\rm a},\mu_{\rm s}, l, \tau, \sigma)$.
\end{theorem}

We wish to use Theorem \ref{Thm.MainInv} to analyze the nonlinear inverse problem \ref{Def.InvProbNonlinear}. The challenge is that the constant of stability $C$ depends on $\mu_{\rm a}$ directly in the upper-right entry in the operator-matrix and indirectly through the operator $\Sop$. There is also dependence on $\tilde{\mu}_{\rm a}$ and $\tilde{w}_{0}$ through $\sigma$. This is a manifestation of the nonlinearity in the recovery of the absorption coefficient. Therefore, in order to address the nonlinear inverse problem \ref{Def.InvProbNonlinear}, we need to determine the dependence of the constant $C$ with respect to changes in $\mu_{\rm a}$, $\tilde{\mu}_{\rm a}$, and $\tilde{w}_{0}$. Our main result for the nonlinear inverse problem is the following.

\begin{theorem}[Main Result 2] \label{Thm.MainInv2}
Concerning the nonlinear inverse problem \ref{Def.InvProbNonlinear},
if both $w_{0}$ and $\tilde{w}_{0}$ are isotropic and either $w$ or $\tilde{w}$ belong to the generic set $\mathcal{U}$ (defined in  Theorem \ref{Thm.MainInv}), then the nonlinear inverse problem is uniquely solvable and the following locally Lipschitz stability estimate,
\begin{eqnarray}
\| \mu_{\rm a} - \tilde{\mu}_{\rm a} \|_{L^{2}(\Omega)} + \| w_{0} - \tilde{w}_{0} \|_{L^{2}(\Omega)} \leq C \| \gamma_{+}w - \gamma_{+}\tilde{w} \|_{H^{1}([0,\tau];\Ts_{+})} \label{Eqn.Stability2}
\end{eqnarray}
holds for a positive constant $C$ that remains uniformly bounded for $\mu_{\rm a}, \tilde{\mu}_{\rm a}$ in a small neighborhood of zero in $L^{\infty}(\Omega)$ and for $w_{0}, \tilde{w}_{0}$ in a small fixed neighborhood of $P_{\theta} \Lss{\infty}_{-}$.
\end{theorem}

In physical terms, the above theorem says that we can recover the absorption coefficient from a \textit{single boundary measurement} if the following conditions are met. First, the initial state of radiation is isotropic (but not necessarily known), that is, we have illuminated the region $\Omega$ during time $t < 0$ as to produce a diffuse state at time $t=0$.
Second, the region $\Omega$ is properly illuminated by the initial state so that radiative particles visit every point in the domain of interest. And third, the medium is weakly absorbing. Notice that the initial state does not have to be fully known a-priori, that is, we do not assume $u_{0} = w |_{t=0} - \tilde{w} |_{t=0} = 0$. In other words, we may uniquely identify the optical absorption even if we use an \textit{unknown}, but sufficiently strong and diffuse, illuminating field. 

The smallness of $\mu_{\rm a}$ is only a technical condition which guarantees the invertibility of (\ref{Eqn.MainSys}). We are not satisfied with this condition because we do not find it to have a physical meaning or relevance. We expect the constant of stability to be very large for strongly absorbing media, but retain the solvability of the inverse problem. Unfortunately, we have not found a way to bypass this assumption at this point.


\section{Proof of the main results} \label{Section:InverseProblem}

In this section, we proceed to prove Theorems \ref{Thm.MainInv}-\ref{Thm.MainInv2}. We reduce the linear inverse problem \ref{Def.InvProb} to an equation of Fredholm form. First, in order to simplify the notation, we introduce the following transport operators $A, A^{*} : \Wss \to \Ws$ given by
\begin{eqnarray}
&& A u = (\theta \cdot \nabla) u + \mu_{\rm a} u + \mu_{\rm s} (I - \Kop) u  \label{Eqn.1200} \\
&& A^{*} \psi =  - (\theta \cdot \nabla) \psi + \mu_{\rm a} \psi + \mu_{\rm s} (I - \Kop^{*}) \psi  \label{Eqn.1202}
\end{eqnarray}
which behave as formal adjoints of each other with respect to the $\Ws$-inner-product.

To make things simple, we momentarily suppose that input data such $\phi$ and $\eta$ in the adjoint (\ref{Eqn.001c})-(\ref{Eqn.003c}) are sufficiently smooth leading to a strong solution $\psi$. Following the usual density arguments, we would take limits and use the continuity of appropriate operators to extend the meaning of the main equations to less regular data. In what follows, we will evaluate the duality pairing between the terms in equation $(\ref{Eqn.0005})$ against $\psi$ and $\dot{\psi}$ to obtain a system of two equations. The system will then be shown to have Fredholm form which is the main strategy of this paper. 

\begin{proof}[Proof of Theorem \ref{Thm.MainInv}]
Let $m = \Lambda(f,u_{0}) \in H^{1}([0,\tau];\Ts_{+})$ and $\sigma_{0} = \sigma |_{t=0}$, and consider
\begin{eqnarray*}
\la \sigma f , \psi \ra_{L^{2}([0,\tau];\Ws)} &=& \la \dot{u} + A u , \psi \ra_{L^{2}([0,\tau];\Ws)}  \\ 
&=& \la u(\tau) , \psi(\tau) \ra_{\Ws} - \la u(0) , \psi(0) \ra_{\Ws} - \la u , \dot{\psi} - A^{*} \psi \ra_{L^{2}([0,\tau];\Ws)} \\
&& \qquad \qquad + \la (\nu \cdot \theta) u , \psi \ra_{L^{2}([0,\tau]; \partial \Omega \times \Sph)} \\ 
&=& - \la u_{0} , \phi \ra_{\Ws} + \la m , \Cop \phi \ra_{L^{2}([0,\tau];\Ts_{+})} 
\end{eqnarray*}
where the boundary term appeared from use of the Green's identity (\ref{Eqn.Green}). The above identity holds for all $\phi \in \Ws$, but we use it only for $\phi \in L^{2}(\Omega)$ because we are considering isotropic source term $f$ and isotropic initial condition $u_{0}$. Hence, we obtain
\begin{eqnarray}
u_{0} + (P_{\theta} P_{t} \sigma \Sop)^{*} f  = P_{\theta} \, \Cop^{*} m, \label{Eqn.Main1}
\end{eqnarray}
where we view $\sigma : L^{2}([0,\tau];\Ws) \to L^{2}([0,\tau];\Ws)$ as a pointwise multiplicative operator mapping $v \mapsto \sigma v$.

Now we proceed to derive a second equation. Consider,
\begin{eqnarray*}
\la \sigma f , \dot{\psi} \ra_{L^{2}([0,\tau];\Ws)} &=& \la \dot{u} + A u , \dot{\psi} \ra_{L^{2}([0,\tau];\Ws)}  \\ 
&=& \la u(\tau) , \dot{\psi}(\tau) \ra_{\Ws} - \la u(0) , \dot{\psi}(0) \ra_{\Ws} - \la u , \ddot{\psi} - A^{*} \dot{\psi} \ra_{L^{2}([0,\tau];\Ws)} \\
&& \qquad \qquad + \la (\nu \cdot \theta) u , \dot{\psi} \ra_{L^{2}([0,\tau]; \partial \Omega \times \Sph)}
\end{eqnarray*}
leading to
\begin{eqnarray*}
\la \sigma_{0} f , \phi \ra_{\Ws} + \la \dot{\sigma} f , \psi \ra_{L^{2}([0,\tau];\Ws)} + \la u_{0} , A^{*} \phi \ra_{\Ws} = \la \dot{m} , \Cop \phi \ra_{L^{2}([0,\tau];\Ts_{+})} 
\end{eqnarray*} 
valid for all sufficiently smooth $\phi \in \Ws$. But again we restrict to all smooth $\phi \in L^{2}(\Omega)$ to obtain,
\begin{eqnarray}
\Big[ (P_{\theta} \sigma_{0}) + (P_{\theta} P_{t} \dot{\sigma} \Sop)^{*} \Big] f + \mu_{\rm a} u_{0}  = P_{\theta} \, \Cop^{*} \dot{m}. \label{Eqn.Main2}
\end{eqnarray}
Here again, the choice of isotropic functions $f$ and $u_{0}$ leads to an advantageous structure for the above equation. In particular, the action of the angular-averaging operator $P_{\theta}$ renders desired compactness (see lemma \ref{Lemma.005} below) as well as the following fact already employed to obtain (\ref{Eqn.Main2}). If $u_{0} , \phi \in L^{2}(\Omega)$ are sufficiently smooth then 
\begin{eqnarray*}
\la u_{0} , A^{*} \phi \ra_{\Ws} = |\Sph| \la u_{0} , P_{\theta} A^{*} \phi \ra_{L^{2}(\Omega)} = |\Sph| \la u_{0} , \mu_{\rm a} \phi \ra_{L^{2}(\Omega)}  .
\end{eqnarray*} 
The last equality is due to $P_{\theta} (\theta \cdot \nabla) \phi = 0$ when $\phi$ is independent of $\theta$, and $P_{\theta}(I-\Kop^{*}) = 0$ due to the conservative nature of the scattering operator $\Kop$. We emphasize that the above equality is a subtle but crucial fact employed in the proof of theorem \ref{Thm.MainInv}.

Equations (\ref{Eqn.Main1}) and (\ref{Eqn.Main2}) constitute the focus of this paper. We already expressed them in operator-valued matrix notation in (\ref{Eqn.MainSys}). Notice that the governing operator of the system (\ref{Eqn.MainSys}) can be expressed as follows,
\begin{eqnarray}
\left[
\begin{array}{cc}
(P_{\theta} \sigma_{0}) & \mu_{\rm a} \\ 
0 & I
\end{array} \right] + 
\left[
\begin{array}{cc}
(P_{\theta} P_{t} \dot{\sigma} \Sop)^{*} & 0 \\ 
(P_{\theta} P_{t} \sigma \Sop)^{*} & 0
\end{array} \right], \label{Eqn.FredForm}
\end{eqnarray}
where the first term is boundedly invertible on $L^{2}(\Omega) \times L^{2}(\Omega)$ provided that $|(P_{\theta} \sigma_{0})| \geq \delta > 0$ and the second term is a compact operator on $L^{2}(\Omega) \times L^{2}(\Omega)$ as asserted by lemma \ref{Lemma.005} below. Hence, we obtain a Fredholm system. 

Now we prove the existence of an open and dense set for $\sigma \in C^{1}([0,\tau];\Ls{\infty}) \cap (i)$ on which (\ref{Eqn.FredForm}) is boundedly invertible. First, standard perturbation shows that the set of $\sigma$'s over which the (\ref{Eqn.FredForm}) is invertible in $L^{2}(\Omega) \times L^{2}(\Omega)$ is open. To show denseness, consider replacing $\sigma$ with
\begin{eqnarray*}
\rho(\lambda) = \lambda \sigma + (1-\lambda) \sigma_{0}.
\end{eqnarray*}
Now notice that the first term in (\ref{Eqn.FredForm}) remains unchanged for any choice of $\lambda \in \mathbb{C}$, and the second term remains compact and analytic with respect to $\lambda \in \mathbb{C}$. If we set $\lambda = 0$, then the governing operator (\ref{Eqn.FredForm}) becomes
\begin{eqnarray*}
\left[
\begin{array}{cc}
(P_{\theta} \sigma_{0}) & \mu_{\rm a} \\ 
(P_{\theta} P_{t} \sigma_{0} \Sop)^{*} & I
\end{array} \right]
\end{eqnarray*}
which is boundedly invertible provided that $\mu_{\rm a}$ is sufficiently small. By the analytic Fredholm theorem \cite{Ren-Rog-2004}, then the system is boundedly invertible for all but a discrete set of $\lambda$'s. In particular, this holds for values arbitrarily close to $\lambda=1$. This shows the desired denseness. The other claims of theorem \ref{Thm.MainInv} are well-known consequences of Fredholm-Riesz-Schauder theory.
\end{proof}

Before going into lemma \ref{Lemma.005}, we wish to make some remarks. Notice that if $f$ and $u_{0}$ were not isotropic, then we would have gotten the following governing operator 
\begin{eqnarray*}
\left[
\begin{array}{cc}
\sigma_{0} + (P_{t} \dot{\sigma} \Sop)^{*} & A \\ 
( P_{t} \sigma \Sop)^{*} & I
\end{array} \right].
\end{eqnarray*}
However, this operator does not have a favorable form. In other words, it is the isotropy of both $f$ and $u_{0}$ what leads to the replacement of $A$ by $\mu_{a}$, and to the appearance of the angular-averaging operator $P_{\theta}$ which renders the needed compactness. If only one of the unknowns $(f,u_{0})$ is assumed isotropic, we do not obtain a favorable structure either as the reader can easily check. In practical applications it is usually acceptable to assume $f$ to be independent of $\theta \in \Sph$. However, assuming that $u_{0}$ is isotropic constitutes the most restrictive condition needed for our approach to work.

Now we proceed to prove a lemma already employed in the proof of Theorem \ref{Thm.MainInv}.

\begin{lemma} \label{Lemma.005}
If $\sigma \in C^{1}([0,\tau];\Ls{\infty})$ then both $P_{\theta} P_{t} \sigma \Sop : L^{2}(\Omega) \to L^{2}(\Omega)$ and $P_{\theta} P_{t} \dot{\sigma} \Sop : L^{2}(\Omega) \to L^{2}(\Omega)$ are compact operators.
\end{lemma}

\begin{proof}
First, we consider $\sigma \in C^{\infty}([0,\tau] \times \clo{\Omega} \times \Sph)$, keeping in mind that this condition will be relaxed later.

Let $\phi \in L^{2}(\Omega)$ and $\eta = \Cop \phi$ and $\psi = \Sop \phi$. We proceed with a density argument by having $\{ \eta^{\epsilon} \}_{\epsilon > 0} \subset C([0,\tau];\Ts_{+})$ be a family of functions such that $\eta^{\epsilon}(\tau) = 0$ and $\eta^{\epsilon} \to \eta$ in the norm of $L^{2}([0,\tau];\Ts_{+})$ as $\epsilon \to 0$. Let also $\psi^{\epsilon} \in C^{1}([0,\tau]; \Ws) \cap C([0,\tau]; \Wss)$ be the unique strong solution of (\ref{Eqn.001c})-(\ref{Eqn.003c}) with $\eta^{\epsilon}$ as the prescribed outflow boundary condition. Notice that $\psi^{\epsilon} \to \psi$ in the norm of $C([0,\tau];\Ws)$ because (\ref{Eqn.001c})-(\ref{Eqn.003c}) is well-posed in a mild sense.

Now, let $\varrho^{\epsilon} = \sigma \psi^{\epsilon} \in C^{1}([0,\tau]; \Ws) \cap C([0,\tau]; \Wss)$. Notice that $\varrho^{\epsilon}$ satisfies (in a strong sense) the following problem,
\begin{eqnarray*}
\dot{\varrho}^{\epsilon} + (\theta \cdot \nabla) \varrho^{\epsilon} = F^{\epsilon} \quad &\text{in $[0,\tau] \times (\Omega \times \Sph)$},  \\
\varrho^{\epsilon} = 0 \quad &\text{on $\{ t = \tau \} \times (\Omega \times \Sph)$},  \\
\gamma_{+} \varrho^{\epsilon} = G^{\epsilon} \quad &\text{on $[0,\tau] \times (\partial \Omega \times \Sph)_{+}$}. 
\end{eqnarray*}
where $G^{\epsilon} = \gamma_{+} (\sigma \psi^{\epsilon}) \in C([0,\tau];\Ts_{+})$ and $F^{\epsilon} = \psi^{\epsilon}( \dot{\sigma} + (\theta \cdot \nabla) \sigma ) + \sigma ( \mu_{\rm a} + \mu_{\rm s}(I-\Kop^{*})) \psi^{\epsilon} \in C([0,\tau];\Ws)$. 

Hence, the time-integral $\varphi^{\epsilon} = P_{t} \varrho^{\epsilon}$ satisfies a stationary problem of the form
\begin{eqnarray*}
(\theta \cdot \nabla) \varphi^{\epsilon} = P_{t} F^{\epsilon}  + \sigma(0) \psi^{\epsilon}(0) \quad &\text{in $ (\Omega \times \Sph)$},  \\
\gamma_{+} \varphi = P_{t} G^{\epsilon} \quad &\text{on $(\partial \Omega \times \Sph)_{+}$}. 
\end{eqnarray*}
The latter is a well-posed stationary adjoint problem (see for instance \cite{Agoshkov-1998,Egg-Sch-2014}) with prescribed outflow condition $P_{t} G^{\epsilon} \in \Ts_{+}$ and forcing term $P_{t}F^{\epsilon} + \sigma(0) \psi^{\epsilon}(0) \in \Ws$. The solution $\varphi^{\epsilon} \in \Wss$ depends continuously on the input data in the appropriate norms. Therefore,
\begin{eqnarray*}
\| \varphi^{\epsilon} \|_{\Wss} \leq C \left( \| \psi^{\epsilon} \|_{C([0,\tau];\Ws)} + \| \eta^{\epsilon} \|_{L^{2}([0,\tau];\Ts_{+})} \right) \leq \tilde{C} \| \eta^{\epsilon} \|_{L^{2}([0,\tau];\Ts_{+})},
\end{eqnarray*}
for all $\epsilon > 0$ where $\eta^{\epsilon} \to \eta = \Cop \phi$ in the norm of $L^{2}([0,\tau];\Ts_{+})$. This in turn implies that the mapping $ \phi \mapsto P_{t} \sigma \Sop \phi$ extends as a bounded operator from $L^{2}(\Omega)$ to $\Wss$. 

Finally, from well-known averaging lemmas \cite{Gol-Lio-Per-Sen-1988}, we obtain that $P_{\theta} P_{t} \sigma \Sop : L^{2}(\Omega) \to H^{1/2}(\Omega)$ is bounded. Our claim follows due to the compact Sobolev embedding of $H^{1/2}(\Omega)$ into $L^{2}(\Omega)$. The above proof is valid for a sufficiently smooth $\sigma$. However, it is straightforward to show that 
\begin{eqnarray*}
\| P_{\theta} P_{t} (\sigma_{1} - \sigma_{2}) \Sop \|_{L^2(\Omega) \to L^{2}(\Omega)} \leq C \| \sigma_{1} - \sigma_{2} \|_{C([0,\tau];\Ls{\infty})}.
\end{eqnarray*}
Since the subspace of compact operators is closed, then we conclude that $P_{\theta} P_{t} \sigma \Sop : L^{2}(\Omega) \to L^{2}(\Omega)$ is compact for $\sigma \in C([0,\tau];\Ls{\infty})$.

The proof of compactness for $P_{\theta} P_{t} \dot{\sigma} \Sop : L^{2}(\Omega) \to L^{2}(\Omega)$ is the same due to our assumption that $\sigma \in C^{1}([0,\tau];\Ls{\infty})$.
\end{proof}

Now we prove our second main result.

\begin{proof}[Proof of Theorem \ref{Thm.MainInv2}]
Suppose, without loss of generality, that $\tilde{w}$ belongs to the generic set $\mathcal{U}$ defined in Theorem \ref{Thm.MainInv}. Also assume that both $w_{0}$ and $\tilde{w}_{0}$ are isotropic. For sufficiently small $\mu_{\rm a} \in L^{\infty}(\Omega)$ we obtain the estimate (\ref{Eqn.Stability}). It only remains to prove the local uniformity of the constant $C$ in this estimate. Recall that the operator (\ref{Eqn.FredForm}) and its inverse depend on $\sigma$ and $\Sop$ which in turn depend on $\tilde{w}_{0}$, $\tilde{\mu}_{\rm a}$ and $\mu_{\rm a}$.

First, we address perturbations of $\sigma$ within the space $C^{1}([0,\tau];\Ls{\infty})$. From the definition in (\ref{Eqn.Diff}), we see that $\sigma$ depends on the evolution operator (c$_0$-semigroup) for the RTE with absorption $\tilde{\mu}_{\rm a}$. Notice that $\tilde{\mu}_{\rm a}$ is a coefficient for a lower-order term in the RTE. Hence, from regular perturbation theory of semigroups in Banach spaces (see \cite[Ch. 3]{Eng-Nag-2000} and \cite[Ch. 9]{Kato-1995}), we obtain that the semigroup operator for the RTE is 
locally Lipschitz continuous with respect to $\tilde{\mu}_{\rm a} \in L^{\infty}(\Omega)$. In particular, we have local Lipschitz continuity of $\sigma \in C^{1}([0,\tau];\Ls{\infty})$ with respect to $\tilde{\mu}_{\rm a} \in L^{\infty}(\Omega)$. Due to linearity and stability of the semigroup, we also have Lipschitz continuity of $\sigma \in C^{1}([0,\tau];\Ls{\infty})$ with respect to its initial condition $\tilde{w}_{0} \in P_{\theta} \Lss{\infty}_{-}$. 

Similarly, the operator $\Sop : \Ws \to L^{2}([0,\tau] ; \Ws)$ is locally Lipschitz continuous with respect to $\mu_{\rm a} \in L^{\infty}(\Omega)$. This follows from the manner in which the control operator $\Cop : \Ws \to L^{2}([0,\tau]; \Ts_{+})$ is constructed. The operator $\Cop$ is given by the bounded inverse of a certain composition of semigroups (see \cite{Aco-2013} for details). In turn, as in the previous paragraph, such semigroups are locally Lipschitz continuous with respect to $\mu_{\rm a} \in L^{\infty}(\Omega)$. Hence, by well-known perturbation arguments \cite{Kato-1995}, we obtain both the locally Lipschitz continuity of $\Cop$ and the invariance of the steering time $\tau$ under small perturbations of the absorption $\mu_{\rm a} \in L^{\infty}(\Omega)$. 

Finally, we have the operator (\ref{Eqn.FredForm}) being locally Lipschitz continuous with respect to perturbations of $\mu_{\rm a}$ and $\tilde{\mu}_{\rm a}$ in $L^{\infty}(\Omega)$, and $\tilde{w}_{0}$ in $P_{\theta} \Lss{\infty}_{-}$. Therefore the same is true for its inverse \cite{Kato-1995}. Let $F : L^{2}(\Omega)^2 \to L^{2}(\Omega)^2$ be the linear operator (\ref{Eqn.FredForm}) and denote its nonlinear dependence on the absorption coefficients and initial condition as follows $F=F[\mu_{\rm a}, \tilde{\mu}_{\rm a},\tilde{w}_{0}]$. Then there exists $\epsilon_{1} > 0$ and $L > 0$ such that
\begin{eqnarray*}
\| F^{-1}[\mu_{\rm a}, \tilde{\mu}_{\rm a} , \tilde{w}_{0}] \| \leq L \left( \| \mu_{a}\|_{L^{\infty}(\Omega)} + \| \tilde{\mu}_{a}\|_{L^{\infty}(\Omega)} \right) + \| F^{-1}[0,0,\tilde{w}_{0}] \|
\end{eqnarray*}
for all $\mu_{\rm a}, \tilde{\mu}_{\rm a} \in L^{\infty}(\Omega)$ satisfying $\left( \| \mu_{a}\|_{L^{\infty}(\Omega)} + \| \tilde{\mu}_{a}\|_{L^{\infty}(\Omega)} \right) < \epsilon_{1}$. 
Now, the dependence of $\sigma$ on $\tilde{w}_{0}$ is linear which implies that $F$ is also linearly and stably dependent on $\tilde{w}_{0} \in P_{\theta} \Lss{\infty}_{-}$. This leads to an additive bounded perturbation of an invertible operator. By well-known arguments \cite{Kato-1995}, there exists $\hat{w}_{0} \in P_{\theta} \Lss{\infty}_{-}$ and $\epsilon_{2} > 0$ such that
\begin{eqnarray*}
 \| F^{-1}[0,0,\tilde{w}_{0}] \| \leq 2  \| F^{-1}[0,0,\hat{w}_{0}] \|
\end{eqnarray*}
for all $\tilde{w}_{0} \in P_{\theta} \Lss{\infty}_{-}$ such that $\| \tilde{w}_{0} - \hat{w}_{0} \|_{\Lss{\infty}} < \epsilon_{2}$. Hence, the constant of stability in (\ref{Eqn.Stability}) can be chosen to be $C = \epsilon_{1} L + 2 \| F^{-1}[0,0,\hat{w}_{0}] \|$ which is independent of $\mu_{\rm a}$ and $\tilde{\mu}_{\rm a}$ in a small neighborhood of zero in $L^{\infty}(\Omega)$, and independent of $w_{0}$ and $\tilde{w}_{0}$ in a small neighborhood of $\hat{w}_{0}$ in $P_{\theta} \Lss{\infty}_{-}$. This provides the local uniformity of the stability constant $C$ in (\ref{Eqn.Stability2}) which concludes the proof. 
\end{proof}


\section*{Acknowledgments}
The author would like to thank Ricardo Alonso for fruitful discussions and the referees for constructive recommendations. This work was partially supported by the AFSOR grant FA9550-12-1-0117 and the ONR grant N00014-12-1-0256.


\bibliographystyle{AIMS}
\bibliography{Biblio}

\medskip
Received xxxx 20xx; revised xxxx 20xx.
\medskip

\end{document}